\documentclass[12pt,reqno]{amsart} 
\usepackage{amssymb,amsmath,amsthm,hyperref} 
\usepackage{a4wide}
\title{Expressing $q$-series in terms of building blocks of Hecke-type double-sums}

\newtheorem{theorem}{Theorem}[section]
\newtheorem{proposition}[theorem]{Proposition}
\newtheorem{corollary}[theorem]{Corollary}
\newtheorem{lemma}[theorem]{Lemma}
\newtheorem{definition}[theorem]{Definition}
\numberwithin{equation}{section}
\DeclareMathOperator{\sg}{sg}
\begin{document}

\author{Eric T. Mortenson}
\address{Department of Mathematics and Computer Science, Saint Petersburg State University, St. Petersburg 199178, Russia}
\email{etmortenson@gmail.com}

\author{Ankit Sahu}
\address{Department of Mathematics and Computer Science, Saint Petersburg State University, St. Petersburg 199178, Russia}
\email{ankitsahu1610@gmail.com}

\date{27 July 2022}



\begin{abstract}
We express recent double-sums studied by Wang, Yee, and Liu in terms of two types of Hecke-type double-sum building blocks.  When possible we determine the (mock) modularity.  We also express a recent $q$-hypergeometric function of Andrews as a mixed mock modular form.
\end{abstract}
\maketitle


\section{Introduction}
We first quickly set some notations:
$$(x)_n=(x;q)_n :=\prod_{i=0}^{n-1}(1-q^{i}x), (x)_{\infty}=(x;q)_\infty := \prod_{i=0}^{\infty}(1-q^{i}x).$$

In recent work, Wang and Yee \cite{WY}, and Liu \cite{L} presented many $q$-series identities involving Hecke-type double-sums. However, they did not address the question of modularity regarding those double-sums. In other words they did not state how their new double-sums relate to Appell--Lerch functions, theta functions, or false theta functions. We define Appell--Lerch functions as follows
\begin{align*}
m(x,z;q):=\frac{1}{j(z;q)}\sum_{r=-\infty}^{\infty}\frac{(-1)^rq^{\binom{r}{2}}z^r}{1-q^{r-1}xz},
\end{align*}
and we define theta functions as
$$j(x;q):=(x)_\infty(q/x)_\infty(q)_\infty = \sum_{n=-\infty}^{\infty}(-1)^nq^{\binom{n}{2}}x^n,$$
where the equality follows from the well-known Jacobi triple product identity. We will also frequently use the following special cases of the above definition
\begin{align*}
J_{a,m}:=j(q^a;q^m), \ \overline{J}_{a,m}:=j(-q^a;q^m),  \text{ and } J_m:=J_{m,3m} = \prod_{i=1}^\infty(1-q^{mi}).
\end{align*}

Hecke-type double-sums have two types of symmetries \cite{CK}. Type I sums are of the form
\begin{equation}
\label{fabc}
f_{a,b,c}(x,y;q) := \left(\sum_{r,s\geq 0} - \sum_{r,s <0}\right)(-1)^{r+s}x^ry^sq^{a\binom{r}{2}+brs+c\binom{s}{2}},
\end{equation}
and type II sums are of the form
\begin{equation}
\label{gabc}
g_{a,b,c}(x,y;q) := \left(\sum_{r,s\geq 0} + \sum_{r,s <0}\right)(-1)^{r+s}x^ry^sq^{a\binom{r}{2}+brs+c\binom{s}{2}}.
\end{equation}
Type I sums were first systematically studied by Hecke \cite{He} but some special cases appeared earlier in the work of Rogers \cite{R}. The modularity questions regarding type I symmetry were solved by Zwegers \cite{Zw}, and Hickerson and Mortenson later demonstrated how to expand these double-sums in terms of Appell--Lerch functions and theta functions \cite{HM}. For example they found expressions such as
\begin{equation}
f_{1,2,1}(x,y;q)=j(y;q)m\big (\frac{q^2x}{y^2},\frac{y}{x};q^3\big )+j(x;q)m\big (\frac{q^2y}{x^2},\frac{x}{y};q^3\big ).\label{equation:f121-example}
\end{equation}
Appell--Lerch functions, $m(x,z;q)$, are also the building blocks of Ramanujan's classical mock-theta functions.

On the other-hand, even though many well-known functions have type II symmetry, very little is known about their general form. For example, the function studied by Andrews, Dyson, and Hickerson  \cite{ADH} reads
\begin{equation*}
\sigma(q)
:=1+\sum_{n=1}^{\infty}\frac{q^{n(n+1)/2}}{(1+q)(1+q^2)\cdots(1+q^n)}
=g_{1,3,3}(-q,q^2;q)-qg_{1,3,3}(-q^3,q^4;q),
\end{equation*}
which is also an example of a quantum modular form \cite{Za1}.
We also have false theta function examples from recent work of Chan and Kim \cite{CK} and Andrews and Warnaar \cite{AW} such as
\begin{equation*}
g_{1,2,2}(q,-q^3;q)=1+2\sum_{n=1}^{\infty}(-1)^nq^{n(n+1)/2}.
\end{equation*}
More such identities can be found in \cite{Mo21}.  Although type II sums appear related to false theta functions similar to how type I sums are related to theta functions, there are no known formulas analogous to (\ref{equation:f121-example}).

Double-sums of types I and II, and Appell--Lerch functions satisfy many interesting properties that are discussed in Section \ref{preliminaries}. This motivates us to try to express unknown double sums in terms of type I or type II symmetries.

Our goal in this work is to express the double-sums of \cite{L, WY} in terms of Hecke-type double-sums of type I symmetry \eqref{fabc} and type II symmetry \eqref{gabc}). We will evaluate the double-sums whenever possible and hopefully shed light on double-sums of type II symmetry. 

As an example, we have equation \eqref{th12} which corresponds to \eqref{WYth1} \cite[Theorem 1.1]{WY}:
\begin{equation*}
\begin{split}
\sum_{n=0}^\infty \sum_{|m|\leq n}(-1)^{m}q^{m^2+n^2}(1-q^{2n+1}) &= f_{1,0,1}(q^2,q^2;q^4)+qf_{1,0,1}(q^4,q^4;q^4)\\
&=j(q^2;q^4).
\end{split}
\end{equation*}
We will give numerous examples where we evaluate such expressions involving double-sums of type I \eqref{fabc} and type II symmetry \eqref{gabc} as described in \cite{Mo21} and \cite{HM}.

\section{Statement of results}

We have the following results from Wang--Yee \cite[Theorem 1.1, 1.2, 1.3]{WY}.
\begin{theorem}
\begin{align}
\sum_{n=1}^\infty \frac{q^n(q;q^2)_n}{(-q;q^2)_n(1+q^{2n})} 
&= \sum_{n=1}^\infty \sum_{|m|\leq n}(-1)^mq^{n^2+m^2}-\sum_{n=1}^\infty (-1)^nq^{2n^2},
\label{WYth1}\\
\sum_{n=1}^\infty \frac{q^{n^2}}{(-q;q^2)_n(1+q^{2n})} 
&= \sum_{n=1}^\infty \sum_{|m|\leq n/2}(-1)^mq^{n^2-2m^2} - \sum_{n=1}^\infty (-1)^nq^{2n^2},
\label{WYth2}\\
\sum_{n=1}^\infty \frac{(-1)^{n-1}q^n(q^2;q^2)_{n-1}}{(-q^2;q^2)_n} 
&= \sum_{n=1}^\infty \sum_{|m|\leq n/2}(-1)^mq^{n^2-2m^2} - \sum_{n=1}^\infty (-1)^nq^{2n^2}.\label{WYth3}
\end{align}
\end{theorem}
\noindent We note that \eqref{WYth2} and \eqref{WYth3} have the same right-hand sides.

Using techniques of Section \ref{calculations} we express \eqref{WYth1} in terms of double-sums of type II symmetry:

\begin{theorem} We have
\label{WYNewTh1}
\begin{equation}
\begin{split}
&\sum_{n=1}^\infty \frac{q^n(q;q^2)_n}{(-q;q^2)_n(1+q^{2n})} = \sum_{n=1}^\infty \sum_{|m|\leq n}(-1)^mq^{n^2+m^2}-\sum_{n=1}^\infty (-1)^nq^{2n^2}\\
&\qquad =\frac{1}{2}\lbrace g_{1,0,1}(q^2,q^2;q^4)+qg_{1,0,1}(q^4,q^4;q^4) + j(q^2;q^4)\rbrace  - \sum_{n=1}^\infty (-1)^nq^{2n^2}\\
&\qquad =\frac{1}{2}\lbrace g_{1,0,1}(q^2,q^2;q^4)+qg_{1,0,1}(q^4,q^4;q^4) + 1\rbrace .
\end{split}
\end{equation}
\end{theorem}


Again using techniques of Section \ref{calculations} we express \eqref{WYth2} and \eqref{WYth3} in terms of double-sums of type II symmetry and obtain
\begin{theorem}
\label{WYNewTh2} We have
\begin{equation}
\begin{split}
&\sum_{n=1}^\infty \frac{q^{n^2}}{(-q;q^2)_n(1+q^{2n})} =\sum_{n=1}^\infty \frac{(-1)^{n-1}q^n(q^2;q^2)_{n-1}}{(-q^2;q^2)_n}\\
&\qquad= \sum_{n=1}^\infty \sum_{|m|\leq n/2}(-1)^mq^{n^2-2m^2} - \sum_{n=1}^\infty (-1)^nq^{2n^2}\\
&\qquad =\frac{1}{2}\lbrace qg_{1,3,1}(q^6,q^6;q^4) + g_{1,3,1}(q^2,q^2;q^4)+q^4g_{1,3,1}(q^{10},q^{10};q^4)\\
&\qquad\qquad+q^9g_{1,3,1}(q^{14},q^{14};q^4)+ j(q^2;q^4)\rbrace - \sum_{n=1}^\infty (-1)^nq^{2n^2}\\
&\qquad =\frac{1}{2}\lbrace qg_{1,3,1}(q^6,q^6;q^4) + g_{1,3,1}(q^2,q^2;q^4)+q^4g_{1,3,1}(q^{10},q^{10};q^4)\\
&\qquad\qquad+q^9g_{1,3,1}(q^{14},q^{14};q^4)+ 1\rbrace.
\end{split}
\end{equation}
\end{theorem}

We now apply the same technique as in Section \ref{calculations} to identities and theorems from a paper by Liu \cite{L}.
\begin{proposition} \cite[Proposition $1.11$]{L}We have
\label{LiuProp1_11}
\begin{align}
\label{Liu1_11a}
\sum_{n=0}^\infty \frac{(q;q^2)_nq^n}{(q^2;q^2)_n}&= \frac{(q;q^2)_{\infty}}{(q^2;q^2)_{\infty}}\sum_{n=0}^\infty\sum_{|m|\leq n}(-1)^{n+m}q^{n^2+n-m^2},\\
\label{Liu1_11b}
\sum_{n=0}^\infty \frac{(-1)^nq^{n^2+n}}{(q^2;q^2)_n}&=\frac{1}{(q;q)_{\infty}}\sum_{n=0}^{\infty}\sum_{|m|\leq n} (-1)^{m}(1-q^{2n+1})q^{2n^2+n-m^2},\\
\label{Liu1_11c}
\sum_{n=0}^\infty \frac{(-1)^nq^{n(n+1)/2}}{(q;q)_n}&=\frac{(-q;q)_\infty}{(q;q)_{\infty}}\sum_{n=0}^{\infty}\sum_{|m|\leq n} (-1)^{m}(1-q^{2n+1})q^{(3n^2+n)/2-m^2}.
\end{align}
\end{proposition}

Using the  techniques in Section \ref{calculations} we express \eqref{Liu1_11a}-\eqref{Liu1_11c} in terms of double-sums of type I symmetry. In each case we see that each of the new expressions further evaluates to a nice product and we obtain nice ``sum = product'' identities. We state them as follows:
\begin{theorem} We have
\label{LiuProp1_11NewTheorem}
\begin{equation}
\label{Liu1_11a_new}
\begin{split}
\sum_{n=0}^\infty \frac{(q;q^2)_nq^n}{(q^2;q^2)_n}&= \frac{(q;q^2)_{\infty}}{(q^2;q^2)_{\infty}}\sum_{n=0}^\infty\sum_{|m|\leq n}(-1)^{n+m}q^{n^2+n-m^2}\\
&= \frac{(q;q^2)_{\infty}}{(q^2;q^2)_{\infty}}f_{0,1,0}(-q,-q;q^4)\\
&=\frac{J_2^2}{J_1},
\end{split}
\end{equation}

\begin{equation}
\label{Liu1_11b_new}
\begin{split}
\sum_{n=0}^\infty \frac{(-1)^nq^{n^2+n}}{(q^2;q^2)_n}&=\frac{1}{(q;q)_{\infty}}\sum_{n=0}^{\infty}\sum_{|m|\leq n} (-1)^{m}(1-q^{2n+1})q^{2n^2+n-m^2}\\
&=\frac{1}{(q;q)_{\infty}}(f_{1,3,1}(q^2,q^2;q^2) + q^3f_{1,3,1}(q^6,q^6;q^2))\\
&=\frac{J_1J_2}{(q;q)_{\infty}}\\
&=J_2,
\end{split}
\end{equation}

\begin{equation}
\label{Liu1_11c_new}
\begin{split}
\sum_{n=0}^\infty \frac{(-1)^nq^{n(n+1)/2}}{(q;q)_n}&=\frac{(-q;q)_\infty}{(q;q)_{\infty}}\sum_{n=0}^{\infty}\sum_{|m|\leq n} (-1)^{m}(1-q^{2n+1})q^{(3n^2+n)/2-m^2}\\
&= \frac{(-q;q)_\infty}{(q;q)_{\infty}}(f_{1,5,1}(q,q;q)+q^2f_{1,5,1}(q^4,q^4;q))\\
&=\frac{(-q;q)_\infty}{(q;q)_{\infty}}J_1J_{1,2}\\
&=J_1.
\end{split}
\end{equation}
\end{theorem}
\noindent
The last equality in \eqref{Liu1_11a_new} is proved in Lemma \ref{lemmaf010}, the penultimate equality in \eqref{Liu1_11b_new} is proved in Section \ref{f131q3f131section}, and the penultimate equality in \eqref{Liu1_11c_new} follows from \cite[pages 219-220]{KP}. The very last equalities in \eqref{Liu1_11b_new} and \eqref{Liu1_11c_new} can be seen very easily.

We remark that identities \eqref{Liu1_11a_new}-\eqref{Liu1_11c_new} have many proofs.  For example, Ae Ja Yee and Jonathan Bradley-Thrush have kindly pointed out that \eqref{Liu1_11a_new}-\eqref{Liu1_11c_new} follow from the $q$-binomial theorem and Euler's identity.  Here we will focus on the robust methods from the setting of Hecke-type double-sums.

Next we have the following equations from Liu  which we state as a proposition:
\begin{proposition} \cite[(4.8), (4.10), (4.6)]{L} We have
\begin{align}
\sum_{n=0}^\infty \frac{q^{n^2}(-q;q)^2_n}{(q;q)_{2n}}
&=\frac{1}{(q;q)_\infty}\sum_{n=0}^{\infty}\sum_{m=0}^{n}(1-q^{6n+6})q^{2n^2+n-m(m+1)/2},\label{Liu4_8}\\
\sum_{n=0}^\infty \frac{q^{n^2}}{(q;q^2)_{n}}
&=\frac{1}{(q;q)_\infty}\sum_{n=0}^{\infty}\sum_{m=0}^{n}(-1)^{n}(1-q^{6n+6})q^{2n^2+n-m(m+1)/2},\label{Liu4_10}\\
\sum_{n=0}^\infty \frac{(-1)^n(q;q^2)_nq^{n^2+n}}{(-q;q)_{2n}}
&=\sum_{n=0}^{\infty} \sum_{m\leq|n|}(-1)^{m+n}(1-q^{4n+2})q^{3n^2+n-m^2}.\label{Liu4_6}
\end{align}
\end{proposition}
To the best of our knowledge, expressions \eqref{Liu4_8} and \eqref{Liu4_10} first appeared in Andrews's \cite[(1.11), (1.10)]{A1}, where the left-hand side of \eqref{Liu4_10} is Ramanujan's third-order mock theta function $\psi(q)$.

We note that the right-hand sides of \eqref{Liu4_8} and \eqref{Liu4_10} are very similar. We again apply the techniques of Section \ref{calculations} and obtain the following three equations respectively:

\begin{theorem} We have
\begin{align}
\sum_{n=0}^\infty \frac{q^{n^2}(-q;q)^2_n}{(q;q)_{2n}}&=\frac{1}{(q;q)_\infty}\sum_{n=0}^{\infty}\sum_{m=0}^{n}(1-q^{6n+6})q^{2n^2+n-m(m+1)/2}\notag\\
&=\frac{1}{(q;q)_\infty}f_{4,4,3}(-q^3,-q^2;q), \label{Liu443-negative}\\
\sum_{n=0}^\infty \frac{q^{n^2}}{(q;q^2)_n}&=\frac{1}{(q;q)_\infty}\sum_{n=0}^\infty\sum_{m=0}^n(-1)^n(1-q^{6n+6})q^{2n^2+n-m(m+1)/2}\notag\\
&=\frac{1}{(q;q)_\infty}f_{4,4,3}(q^3,q^2;q), \label{Liu443-positive}\\
\sum_{n=0}^\infty \frac{(-1)^n(q;q^2)_nq^{n^2+n}}{(-q;q)_{2n}}&=\sum_{n=0}^{\infty} \sum_{m\leq|n|}(-1)^{m+n}(1-q^{4n+2})q^{3n^2+n-m^2}\notag\\
&=g_{1,2,1}(-q^{3},-q^{3};q^4)-q^4g_{1,2,1}(-q^{9},-q^{9};q^4).
\end{align}
\end{theorem}
\noindent In Section \ref{section:Andrews-mock}, we will show that $f_{4,4,3}( q^3, q^2;q)$ is a known mock theta function and that $f_{4,4,3}( -q^3, -q^2;q)$ is a new mixed mock modular form.  Indeed, we will see that  Andrews' \cite[(1.11)]{A1} can be expressed
\begin{equation}
\sum_{n=0}^\infty \frac{q^{n^2}(-q;q)^2_n}{(q;q)_{2n}}
    =\frac{1}{J_{1}}f_{4,4,3}(-q^3,-q^2;q)
    =\frac{1}{4}\frac{\overline{J}_{0,3}}{J_1}\mu(q^3)
    -\frac{1}{2}\frac{\overline{J}_{1,3}}{J_1}\phi(q)+\frac{1}{J_1}\Theta(q),
\end{equation}
where $\Theta(q)$ is a sum of quotients of theta functions, $\mu(q)$ is a second-order mock theta function, and $\phi(q)$ is a sixth-order mock theta function.  For more information on the respective mock theta functions, we refer the interested reader to \cite[Section $5$]{HM} and references therein.

\section{Preliminaries}
\label{preliminaries}
In this section we collect known properties of Hecke-type double-sums, theta functions, and Appell--Lerch functions that will be used in the proofs of our results.  The following four properties satisfied by double-sums of type I and type II symmetry are given in \cite{HM, Mo21}:

\begin{proposition} \cite[Proposition $6.1$]{HM} We have
\label{f1}
$$f_{a,b,c}(x,y;q) = -\frac{q^{a+b+c}}{xy}f_{a,b,c}(q^{2a+b}/x,q^{2c+b}/y;q).$$
\end{proposition}

\begin{proposition} \cite[Proposition $6.3$]{HM}  We have
\label{f2}
\begin{equation*}
\begin{split}
f_{a,b,c}(x,y;q) = (-x)^l(-y)^kq^{a\binom{l}{2}+blk+c\binom{k}{2}}f_{a,b,c}(q^{al+bk}x,q^{bl+ck}y;q)\\
+\sum_{m=0}^{l-1}(-x)^mq^{a\binom{m}{2}}j(q^{mb}y;q^c)+\sum_{m=0}^{k-1}(-y)^mq^{c\binom{m}{2}}j(q^{mb}x;q^{a}).
\end{split}
\end{equation*}
\end{proposition}

\begin{proposition} \cite[Proposition $2.4$]{Mo21} We have
\label{g1}
$$g_{a,b,c}(x,y;q) = \frac{q^{a+b+c}}{xy}g_{a,b,c}(q^{2a+b}/x,q^{2c+b}/y;q).$$
\end{proposition}

\begin{proposition} \cite[Proposition $2.5$]{Mo21} We have
\label{g2}
\begin{equation*}
\begin{split}
g_{a,b,c}(x,y;q) = (-x)^l(-y)^kq^{a\binom{l}{2}+blk+c\binom{k}{2}}g_{a,b,c}(q^{al+bk}x,q^{bl+ck}y;q)\\
+\sum_{r=0}^{l-1}(-x)^rq^{a\binom{r}{2}}\sum_{s \in \mathbb{Z}}\sg(s)(-y)^sq^{brs}q^{c\binom{s}{2}},
\end{split}
\end{equation*}
where
\begin{equation*}
\sg(x):=\begin{cases}1 &\text{if}\ x\geq 0,\\
-1 &\text{if}\ x<1.
\end{cases}
\end{equation*}
\end{proposition}
Next we mention a useful theorem from \cite{HM}.
\begin{theorem} \cite[Proposition 8.1]{HM}
\label{f131decomposition}
Let $l\in \mathbf{Z}$ and $x,y \in \mathbf{C}\backslash \{ 0 \}$. Then we have
\begin{equation*}
\begin{split}
f_{1,3,1}(x,y;q) = j(y;q)m(-q^5x/y^3,q^{2l}y/x;q^8)
+j(x;q)m(-q^5y/x^3,x/q^{2l}y;q^8)\\
-(-1)^l\frac{q^{4l+1+\binom{l}{2}}x^{l+1}yJ_{2,4}J_{8,16}j(q^{4l+3}xy;q^8)j(q^{8l+14}x^2y^2;q^{16})}{j(-q^{2l+3}x^2;q^8)j(-q^{6l+3}y^2;q^8)}.
\end{split}
\end{equation*}
\end{theorem}

We will also need the following equation from \cite{HM}, which we state here for convenience as a theorem
\begin{theorem} \cite[(2.2f)]{HM} We have
$$j(z;q)=\sum_{k=0}^{m-1}(-1)^kq^{\binom{k}{2}}z^kj((-1)^{m+1}q^{\binom{m}{2}+mk}z^{m};q^{m^2}).$$
\end{theorem}
\noindent For $m=2$ we get the following corollary:
\begin{corollary}
\begin{equation}
\label{jinto2sums}
j(z;q)=j(-qz^2;q^4) - zj(-q^3z^2;q^4).
\end{equation}
\end{corollary}

Next we need the following properties of the Appell--Lerch function $m(x,z;q)$ (defined in the introduction):

\begin{proposition} \cite[Proposition $3.1$]{HM} We have
\begin{subequations}
\begin{align}
\label{mrelation1}
m(x,z;q) &= x^{-1}m(x^{-1},z^{-1};q),\\
\label{mrelation2}
m(qx,z;q) &=1-xm(x,z;q).
\end{align}
\end{subequations}
\end{proposition}

Here is an identity known to Kronecker mentioned in \cite{Kr1}, \cite[pp. 309-318]{Kr2}:

\begin{proposition}We have
\label{KroneckerIdentity}
\begin{equation}
\left(\sum_{r,s\ge 0} - \sum_{r,s<0}\right)q^{rs}x^ry^s = \frac{(q)^2_{\infty}(xy,q/xy;q)_\infty}{(x,q/x,y,q/y;q)_\infty}.
\end{equation}
\end{proposition}
\noindent
Note that the left-hand side is nothing but $f_{0,1,0}(-x,-y;q)$.

We will also have a brief occasion to use the following identities for theta functions,  which we state as a proposition
\begin{proposition} \cite[(2.2b), (2.2c)]{HM} We have
\label{jIdentities}
\begin{subequations}
\begin{align}
\label{jIdentityA}
j(x;q) = j(q/x;q),\\
\label{jIdentityB}
j(qx;q) = -x^{-1}j(x;q),
\end{align}
\end{subequations}
\begin{equation}
\label{jx2Identity}
j(x^2;q^2) = j(x;q)j(-x;q)\frac{J_2}{J_1^2} = j(x;q)j(-x;q)\frac{(q^2;q^2)_\infty}{(q;q)_{\infty}^2},
\end{equation}
\begin{equation}
\label{jxToProduct}
j(x;q) = j(x;q^2)j(qx;q^2)\frac{J_1}{J_2^2}.
\end{equation}
\end{proposition}
These equations are very easy to prove directly from the definition itself. We prove \eqref{jxToProduct} which is very similar to \cite[(2.2e)]{HM}.  We write 
$$j(x;q) = (x;q)_{\infty}(q/x;q)_{\infty}(q;q)_{\infty}$$
and then use 
$$(x;q)_{\infty} = (x;q^2)_{\infty}(qx;q^2)_{\infty}$$
for the first two terms and obtain
$$j(x;q) = (x;q^2)_{\infty}(qx;q^2)_{\infty}(q/x;q^2)_{\infty}(q^2/x;q^2)_{\infty}(q;q)_{\infty}\frac{(q^2;q^2)_{\infty}^2}{(q^2;q^2)_{\infty}^2},$$
which can easily be seen to be equal to
$$j(x;q^2)j(qx;q^2)\frac{J_1}{J_2^2}.$$

We will also use the following identities later which we collect in the following proposition:
\begin{proposition} \cite[Section 2]{HM} We have
\label{bigJidentities}
\begin{align*}
J_{1,2} = \frac{J_1^2}{J_2},\ J_{1,4} =\frac{J_1J_4}{J_2},\
\overline{J}_{1,4} = \frac{J_2^2}{J_1}.
\end{align*}
\
\end{proposition}

\section{Evaluation of Hecke-sums in terms of double-sums of type I and type II symmetry}
\label{calculations}
The basic technique to write Hecke-type sums in terms of double-sums with type I or type II symmetry is the same---we apply a suitable change of variables and then apply suitable $q$-series techniques to convert the sum into $f$'s and $g$'s. We demonstrate this technique for the double-sums in \eqref{th21} (or \eqref{th22}) which we believe is slightly more involved than others.

As an example, let us begin with
\begin{equation}
\label{th2expression}
\sum_{n=0}^\infty \sum_{|m|\leq n/2}(-1)^mq^{n^2-2m^2}.
\end{equation}
Notice that we start $n$ from $0$ instead of $1$.  We first split $n$ with respect to $n$ even or $n$ odd (written $2n$ or $2n+1$). This yields
\begin{equation}
\label{sum2}
\sum_{n=0}^\infty \sum_{|m|\leq n}(-1)^mq^{4n^2-2m^2}+ \sum_{n=0}^\infty \sum_{|m|\leq n}(-1)^mq^{4n^2+4n+1-2m^2}.
\end{equation}
We next apply the transformations
\begin{align*}
r &= n+m \geq 0,\\
s &= n-m \geq 0,\\
n &=\frac{r+s}{2}, &m = \frac{r-s}{2},
\end{align*}
and then we apply the condition on whether $(r,s) = (2R,2S)$ or $(2R+1,2S+1)$.
We then obtain the following sum of four series for our original series (where we have again re-written ($R$, $S$) as ($r$,$s$) for better typography):
\begin{align*}
\sum_{r,s\geq 0}(-1)^{r+s}q^{4\binom{r}{2}+4\binom{s}{2}+12rs} q^{2r} q^{2s}\\
&+ \sum_{r,s\geq 0}(-1)^{r+s}q^{4\binom{r}{2}+4\binom{s}{2}+12rs} q^{6r} q^{6s} q\\
+\sum_{r,s\geq 0}(-1)^{r+s}q^{4\binom{r}{2}+4\binom{s}{2}+12rs} q^{10r} q^{10s} q^{4}\\
&+ \sum_{r,s\geq 0}(-1)^{r+s}q^{4\binom{r}{2}+4\binom{s}{2}+12rs}  q^{14r} q^{14s} q^9.
\end{align*}

Let us denote the four series above as follows:
\begin{align*}
&A1 &+& &B1\\
&+C1 &+& &D1.
\end{align*}
The two series on the left ($A1$ and $C1$) correspond to the left term in \eqref{sum2} ($A1$ corresponds to $(r,s) = (2R,2S)$, and $C1$ to $(r,s) = (2R+1,2S+1)$). Similarly $B1$, $D1$ correspond to the right term in \eqref{sum2}.

We then evaluate the following two series (that we get by opening the brackets), similarly, and get eight terms (four for each series)

\begin{equation}
\label{section4temp1}
\sum_{n=0}^\infty \sum_{|m|\leq n/2}(-1)^mq^{n^2-2m^2}(1+q^{\alpha m +\beta n + k}).
\end{equation}
We obtain final expressions in which exponents are in terms of $\alpha$, $\beta$, and $k$. We then select these $\alpha$, $\beta$, and $k$ carefully ($(\alpha,\beta,k) = (0,2,1)$ respectively) so that the eight series complement one another so as to form $f$'s and $g$'s.

Note that \eqref{section4temp1} is the same as the series \eqref{th2expression} multiplied by $(1+q^{\alpha m +\beta n +k})$ so the first four terms are same as $A1$, $B1$, $C1$, $D1$ above. The other four terms that we get are the following where we have applied the transformation $r\leftarrow -r-1$ and $s\leftarrow -s-1$ and changed $\sum_{r,s \geq 0}$ to $\sum_{r,s <0}$.  The four terms are

\begin{align*}
\sum_{r,s < 0}(-1)^{r+s}q^{4\binom{r}{2}+4\binom{s}{2}+12rs} q^{14r} q^{14s} q^9\\
&+\sum_{r,s < 0}(-1)^{r+s}q^{4\binom{r}{2}+4\binom{s}{2}+12rs} q^{10r} q^{10s} q^4\\
+\sum_{r,s < 0}(-1)^{r+s}q^{4\binom{r}{2}+4\binom{s}{2}+12rs} q^{6r} q^{6s}q\\
&+\sum_{r,s < 0}(-1)^{r+s}q^{4\binom{r}{2}+4\binom{s}{2}+12rs}  q^{2r}q^{2s}.
\end{align*}
We denote these four terms respectively
\begin{align*}
&A2 &+& &B2\\
&+C2 &+& &D2.
\end{align*}

For more clarity we show how the terms $B1$ and $D1$ combine to form   $qg_{1,3,1}(q^6,q^6;q^4)$. After applying the transformations $r\leftarrow -r-1$ and $s\leftarrow -s-1$ to $D1$ we get the following:

$$D1 = \sum_{r,s<0}(-1)^{r+s}q^{4\binom{r}{2}+4\binom{s}{2}+12rs} \cdot q^{6r}\cdot q^{6s}\cdot q.$$
We can now see 
$$B1 + D1 = qg_{1,3,1}(q^6,q^6;q^4).$$
Similarly, we obtain 
\begin{align*}
A1 + D2 &= g_{1,3,1}(q^2,q^2;q^4),\\
A2 +C2 &= q^9g_{1,3,1}(q^{14},q^{14};q^4),\\
C1+B2 &= q^4g_{1,3,1}(q^{10},q^{10};q^4).
\end{align*}
Finally we obtain the following:
\begin{equation}
\begin{split}
\sum_{n=0}^\infty\sum_{|m|\leq n/2}(-1)^mq^{n^2-2m^2}(1+q^{2n+1}) = qg_{1,3,1}(q^6,q^6;q^4) + g_{1,3,1}(q^2,q^2;q^4)\\
+q^4g_{1,3,1}(q^{10},q^{10};q^4)+q^9g_{1,3,1}(q^{14},q^{14};q^4),
\end{split}
\end{equation}
and

\begin{equation}
\begin{split}
\sum_{n=0}^\infty\sum_{|m|\leq n/2}(-1)^mq^{n^2-2m^2}(1-q^{2n+1}) = qg_{1,3,1}(q^6,q^6;q^4) + f_{1,3,1}(q^2,q^2;q^4)\\+q^4f_{1,3,1}(q^{10},q^{10};q^4)-q^9g_{1,3,1}(q^{14},q^{14};q^4).
\end{split}
\end{equation}

We would finally like to mention that for the sums of the form
$$\sum_{n=0}^{\infty}\sum_{m=0}^{n},$$
which appears in \eqref{Liu443-negative}, we use the following change of variables
\begin{align*}
n=r+s,\ m=s.
\end{align*}

\section{Proofs}
\subsection{Proof of Theorem \ref{WYNewTh1}}

We apply the techniques of Section \ref{calculations} to the double summation
$$\sum_{n=1}^\infty \sum_{|m|\leq n}(-1)^mq^{n^2+m^2}$$
in the right-hand side of \eqref{WYth1} and get the following

\begin{proposition}We have
\begin{equation}
\label{th11}
\sum_{n=0}^\infty \sum_{|m|\leq n}(-1)^{m}q^{m^2+n^2}(1+q^{2n+1}) = g_{1,0,1}(q^2,q^2;q^4)+qg_{1,0,1}(q^4,q^4;q^4),
\end{equation}
and
\begin{equation}
\label{th12}
\begin{split}
\sum_{n=0}^\infty \sum_{|m|\leq n}(-1)^{m}q^{m^2+n^2}(1-q^{2n+1}) &= f_{1,0,1}(q^2,q^2;q^4)+qf_{1,0,1}(q^4,q^4;q^4)\\
&=j(q^2;q^4).
\end{split}
\end{equation}
\end{proposition}
The second equality in \eqref{th12} follows from Lemma \ref{lemmaf0} and Lemma \ref{lemmaftoj}. Next we note that adding \eqref{th11} and \eqref{th12} produces the first term in the right-hand side of \eqref{WYth1} proving Theorem \ref{WYNewTh1}.  For the last equality we use the identity
\begin{equation}
    j(q^2;q^4)=1+\sum_{n=1}^{\infty}(-1)^{q^{2n^2}}.\label{equation:last-id}
\end{equation}

\subsection{Proof of Theorem \ref{WYNewTh2}}
Again we apply the techniques of Section \ref{calculations} to the double summation in the right-hand side of \eqref{WYth2} and we get the following
\begin{proposition}We have
\begin{equation}
\label{th21}
\begin{split}
\sum_{n=0}^\infty\sum_{|m|\leq n/2}(-1)^mq^{n^2-2m^2}(1+q^{2n+1}) = qg_{1,3,1}(q^6,q^6;q^4) + g_{1,3,1}(q^2,q^2;q^4)\\+q^4g_{1,3,1}(q^{10},q^{10};q^4)+q^9g_{1,3,1}(q^{14},q^{14};q^4),
\end{split}
\end{equation}
and
\begin{equation}
\label{th22}
\begin{split}
\sum_{n=0}^\infty\sum_{|m|\leq n/2}(-1)^mq^{n^2-2m^2}(1-q^{2n+1}) &= qg_{1,3,1}(q^6,q^6;q^4) + f_{1,3,1}(q^2,q^2;q^4)\\&+q^4f_{1,3,1}(q^{10},q^{10};q^4)-q^9g_{1,3,1}(q^{14},q^{14};q^4)\\
&=f_{1,3,1}(q^2,q^2;q^4) =j(q^2;q^4).
\end{split}
\end{equation}
\end{proposition}

Where we have used $f_{1,3,1}(q^{10},q^{10};q^4) =0$ (Lemma \ref{lemmaf0_2}), and $$qg_{1,3,1}(q^6,q^6;q^4)-q^9g_{1,3,1}(q^{14},q^{14};q^4) =0$$ (Lemma \ref{lemmag0}). The double-sum $f_{1,3,1}(q^2,q^2;q^4)$ is evaluated in Section \ref{f131Thsection} (Theorem \ref{f131Th}). Adding \eqref{th21} and \eqref{th22} produces Theorem \ref{WYNewTh2}.  For the last equality we use the identity (\ref{equation:last-id}).

\subsection{Proof of Theorem \ref{LiuProp1_11NewTheorem}}
Corresponding to the three equations in Proposition \ref{LiuProp1_11} we get the following after applying technique of Section \ref{calculations}

\begin{equation}
\label{Liu1_11aIntermediate}
\sum_{n=0}^\infty\sum_{|m|\leq n}(-1)^{n+m}q^{n^2+n-m^2} = f_{0,1,0}(-q,-q;q^4),
\end{equation}

\begin{equation}
\label{Liu1_11bIntermediate}
\sum_{n=0}^{\infty}\sum_{|m|\leq n} (-1)^{m}(1-q^{2n+1})q^{2n^2+n-m^2} = f_{1,3,1}(q^2,q^2;q^2) + q^3f_{1,3,1}(q^6,q^6;q^2),
\end{equation}
and
\begin{equation}
\label{Liu1_11cIntermediate}
\sum_{n=0}^{\infty}\sum_{|m|\leq n} (-1)^{m}(1-q^{2n+1})q^{(3n^2+n)/2-m^2} = f_{1,5,1}(q,q;q)+q^2f_{1,5,1}(q^4,q^4;q).
\end{equation}

We will prove in Section \ref{f131q3f131section} that 
$$f_{1,3,1}(q^2,q^2;q^2) + q^3f_{1,3,1}(q^6,q^6;q^2) = J_1J_2.$$ Combined with this fact the above three equations readily imply Theorem \ref{LiuProp1_11NewTheorem}.
\section{Reduction and cancellation of double-sums in the final expressions}
\label{reductions}
\begin{lemma}We have
\label{lemmaf0}
$$f_{1,0,1}(q,q;q) =0.$$
\end{lemma}
\begin{proof}
This is a straight forward application of Proposition \ref{f1} from which we get
$$f_{1,0,1}(q,q;q)= - f_{1,0,1}(q,q;q),$$
implying $f_{1,0,1}(q,q;q) =0$.
\end{proof}

\begin{lemma}We have
\label{lemmaftoj}
$$f_{1,0,1}(q,q;q^2) = j(q;q^2).$$
\end{lemma}
\begin{proof}
We use Proposition \ref{f1} and \ref{f2} both.
We first apply Proposition \ref{f2} with $l=k=1$ to get
$$f_{1,0,1}(q,q;q^2) = q^2f_{1,0,1}(q^3,q^3;q^2)+2j(q;q^2).$$
Then we apply Proposition \ref{f1} to get $q^2f_{1,0,1}(q^3,q^3;q^2)=-f_{1,0,1}(q,q;q^2)$, from which we get

\begin{equation*}
f_{1,0,1}(q,q;q^2) = -f_{1,0,1}(q,q;q^2)+2j(q;q^2).
\end{equation*}
Adding $f_{1,0,1}(q,q;q^2)$ to both sides gives
$$f_{1,0,1}(q,q;q^2) = j(q;q^2).$$
\end{proof}

\begin{lemma}We have
\label{lemmaf0_2}
$$f_{1,3,1}(q^{5},q^{5};q^2) = 0.$$
\end{lemma}
\begin{proof}
Again this is a straightforward application of Proposition \ref{f1} from which we get
$$f_{1,3,1}(q^{5},q^{5};q^2) = -f_{1,3,1}(q^{5},q^{5};q^2),$$
which gives $f_{1,3,1}(q^{5},q^{5};q^2) =0$.
\end{proof}

\begin{lemma}We have
\label{lemmag0}
$$g_{1,3,1}(q^3,q^3;q^2) - q^4g_{1,3,1}(q^{7},q^{7};q^2) =0.$$
\end{lemma}
\begin{proof}
This is an easy application of Proposition \ref{g1}. Proposition \ref{g1} gives
$$g_{1,3,1}(q^3,q^3;q^2) = q^4g_{1,3,1}(q^{7},q^{7};q^2).$$
\end{proof}

\begin{lemma}We have
\label{lemmaf010}
$$\frac{(q;q^2)_{\infty}}{(q^2;q^2)_{\infty}}f_{0,1,0}(-q,-q;q^4)
=\frac{J_2^2}{J_1}.$$
\end{lemma}
\begin{proof}
Applying Proposition \ref{KroneckerIdentity} to $f_{0,1,0}(-q,-q;q^4)$ yields
$$f_{0,1,0}(-q,-q;q^4) = \frac{(q^4;q^4)_{\infty}^2(q^2;q^4)_{\infty}^2}{(q;q^4)_{\infty}^2(q^3;q^4)_{\infty}^2},$$
and we can write $$(q;q^2)_\infty = (q;q^4)_\infty(q^3;q^4)_\infty$$ and $$(q^2;q^2)_\infty = (q^2;q^4)_\infty(q^4;q^4)_\infty.$$ Simple calculations then yield the following

\begin{align*}&\frac{(q;q^2)_{\infty}}{(q^2;q^2)_{\infty}}f_{0,1,0}(-q,-q;q^4)
\\&=\frac{(q^4;q^4)_\infty(q^2;q^4)_\infty}{(q;q^4)_\infty(q^3;q^4)_\infty} =\frac{(q^4;q^4)_\infty^2(q^2;q^4)_\infty^2}{(q;q^4)_\infty(q^3;q^4)_\infty(q^4;q^4)_\infty(q^2;q^4)_\infty}\\
&=\frac{\prod_{i=1}^\infty (1-q^{2i})^2}{\prod_{i=1}^\infty(1-q^{i})} = \frac{J_2^2}{J_1}.\qedhere
\end{align*}

\end{proof}

\section{The double-sum $f_{1,3,1}(q^2,q^2;q^4)$}
\label{f131Thsection}
This section is devoted to proving $f_{1,3,1}(q^2,q^2;q^4) = j(q^2;q^4)$. We use Proposition 6.1 from (\cite{HM}) which reads as follows:

\begin{proposition}For $x, y \in \mathbf{C}\backslash \{ 0 \}$
\begin{equation*}
\begin{split}
f_{a,b,c}(x,y;q) = &f_{a,b,c}(-x^2q^{a},-y^2q^{c};q^4)-xf_{a,b,c}(-x^2q^{3a},-y^2q^{c+2b};q^4)\\
&-yf_{a,b,c}(-x^2q^{a+2b},-y^2q^{3c};q^4)\\&+xyq^bf_{a,b,c}(-x^2q^{3a+2b},-y^2q^{3c+2b};q^4).
\end{split}
\end{equation*}
\end{proposition}

The proposition can be proved easily by decomposing the definition of $f_{a,b,c}(x,y;q)$ according to the parity of $r,s$.

From this we obtain the following two equations
\begin{equation*}
\begin{split}
f_{1,3,1}(q^{1/2},q^{1/2};-q) = &f_{1,3,1}(q^2,q^2;q^4)-q^{1/2}f_{1,3,1}(q^4,q^8;q^4)\\
&-q^{1/2}f_{1,3,1}(q^8,q^4;q^4)-q^4f_{1,3,1}(q^{10},q^{10};q^4),
\end{split}
\end{equation*}
and
\begin{equation*}
\begin{split}
f_{1,3,1}(-q^{1/2},-q^{1/2};-q) = &f_{1,3,1}(q^2,q^2;q^4)+q^{1/2}f_{1,3,1}(q^4,q^8;q^4)\\
&+q^{1/2}f_{1,3,1}(q^8,q^4;q^4)-q^4f_{1,3,1}(q^{10},q^{10};q^4).
\end{split}
\end{equation*}
As noted before $f_{1,3,1}(q^{10},q^{10};q^4)=0$ (Lemma \ref{lemmaf0_2}). Now adding the two equations above we get
\begin{equation}
2f_{1,3,1}(q^2,q^2;q^4)=f_{1,3,1}(q^{1/2},q^{1/2};-q)+f_{1,3,1}(-q^{1/2},-q^{1/2};-q).
\end{equation}
Now we apply Theorem \ref{f131decomposition} to the two terms in the right-hand side separately. We note that if we take $l$ to be odd, then the last quotient in Theorem \ref{f131decomposition} from the two terms simply cancel out yielding
\begin{equation}
\label{after95}
\begin{split}
2f_{1,3,1}(q^2,q^2;q^4)=&j(q^{1/2};-q)m(q^4,q^{2l};q^8)
+j(q^{1/2};-q)m(q^4,q^{-2l};q^8)\\
&+j(-q^{1/2};-q)m(q^4,q^{2l};q^8)+j(-q^{1/2};-q)m(q^4,q^{-2l};q^8).
\end{split}
\end{equation}
From \eqref{jinto2sums} we obtain the following two equations:
\begin{align*}
&j(q^{1/2};-q) = j(q^2;q^4)-q^{1/2}j(q^4;q^4),\\
&j(-q^{1/2};-q) = j(q^2;q^4) + q^{1/2}j(q^4;q^4).
\end{align*}
Since $j(q^4;q^4) =0$ we get
\begin{equation}
j(q^{1/2};-q) =j(-q^{1/2};-q)= j(q^2;q^4).
\end{equation}
So equation \eqref{after95} reduces to
\begin{equation}
\label{fprefinal}
f_{1,3,1}(q^2,q^2;q^4)=j(q^{1/2};-q)(m(q^4,q^{2l};q^8)+m(q^4,q^{-2l};q^8)).
\end{equation}

Now we work with $m(q^4,q^{-2l};q^8)$. The Appell--Lerch function $m$ satisfies some basic relations as given in \eqref{mrelation1} and \eqref{mrelation2}. Applying \eqref{mrelation1} and then \eqref{mrelation2} we get the following
\begin{align*}
m(q^4,q^{-2l};q^8) &= q^{-4}m(q^{-4},q^{2l};q^8)\\
&= 1 - m(q^4,q^{2l};q^8).
\end{align*}
Where we have used \eqref{mrelation2} in the form $xm(x,z;q) =1-m(qx,z;q)$.
Using this in \eqref{fprefinal} yields

\begin{theorem}
\label{f131Th}
$$f_{1,3,1}(q^2,q^2;q^4) = j(q^{1/2};-q) = j(-q^{1/2};-q) = j(q^2;q^4).$$
\end{theorem}

\section{The double-sum $f_{1,3,1}(q^2,q^2;q^2) + q^3f_{1,3,1}(q^6,q^6;q^2)$}
\label{f131q3f131section}
In this section we evaluate the expression 
$$f_{1,3,1}(q^2,q^2;q^2) + q^3f_{1,3,1}(q^6,q^6;q^2).$$ We apply Theorem \ref{f131decomposition} to both double-sums. The first two terms in the right-hand side of Theorem \ref{f131decomposition} are 
$$j(y;q)m(-q^5x/y^3,q^{2l}y/x;q^8)$$ and 
$$j(x;q)m(-q^5y/x^3,x/q^{2l}y;q^8).$$ We note that for both $f_{1,3,1}(q^2,q^2;q^2)$ and $f_{1,3,1}(q^6,q^6;q^2)$, both the theta functions, $j$, are $0$. If we choose $l\not\equiv 0\ (\textrm{mod}\ 4)$ in the theorem the $m$ terms won't have any poles and hence the two terms can be ignored. We take $l=-1$ and get the following

\begin{equation}
\begin{split}
&f_{1,3,1}(q^2,q^2;q^2)+ q^3f_{1,3,1}(q^6,q^6;q^2)\\
&=\frac{q^{-2}J_{4,8}J_{16,32}j(q^2;q^{16})j(q^{20};q^{32})}{j(-q^{6};q^{16})j(-q^{-2};q^{16})} + \frac{q^{5}J_{4,8}J_{16,32}j(q^{10};q^{16})j(q^{36};q^{32})}{j(-q^{14};q^{16})j(-q^{6};q^{16})}.
\end{split}
\end{equation}
Now we want to simplify this expression. We first note that the denominators differ only by a $q$ power as follows. Equation \eqref{jIdentityB} gives us the following

$$j(-q^{16}\cdot q^{-2};q^{16}) = q^2j(-q^{-2};q^{16}),$$
which gives
$$j(-q^{-2};q^{16}) = q^{-2}j(-q^{14};q^{16}).$$
And similarly we obtain $j(q^{36};q^{32}) = q^{-4}j(q^4;q^{32})$ in the second numerator.
Now we use \eqref{jx2Identity} for $j(q^{20};q^{32})$ and $j(q^4;q^{32})$
\begin{align*}
j(q^{20};q^{32}) = j(q^{10};q^{16})j(-q^{10};q^{16})\frac{(q^{32};q^{32})_\infty}{(q^{16};q^{16})_{\infty}^2},\\
j(q^{4};q^{32}) = j(q^{2};q^{16})j(-q^{2};q^{16})\frac{(q^{32};q^{32})_\infty}{(q^{16};q^{16})_{\infty}^2}.
\end{align*}
And finally we get the following:
\begin{equation}
\label{f131intermediate}
\frac{J_{4,8}J_{16,32}j(q^2;q^{16})j(q^{10};q^{16})(q^{32};q^{32})_\infty}{j(-q^6;q^{16})j(-q^{14};q^{16})(q^{16};q^{16})_\infty^2}(j(-q^{10};q^{16}) - qj(-q^2;q^{16})).
\end{equation}
Next we use \eqref{jinto2sums} to see that
$$j(-q^2;q^{16}) - q^{-1}j(-q^{10};q^{16}) = j(q^{-1};q^{4}),$$
which gives
$$j(-q^{10};q^{16}) - qj(-q^2;q^{16}) = -qj(q^{-1};q^{4}).$$
Thus we obtain the final expression from \eqref{f131intermediate} as follows
$$\frac{-qJ_{4,8}J_{16,32}j(q^2;q^{16})j(q^{10};q^{16})(q^{32};q^{32})_{\infty}j
(q^{-1};q^4)}{j(-q^6;q^{16})j(-q^{14};q^{16})(q^{16};q^{16})_\infty^2},$$
which we can rewrite as
$$\frac{-qJ_{4,8}J_{16,32}j(q^2;q^{16})j(q^{10};q^{16})J_{32}j
(q^{-1};q^4)}{j(-q^6;q^{16})j(-q^{14};q^{16})J_{16}^2}.$$
From \eqref{jIdentityB} we obtain $-qj(q^{-1};q^4) = j(q^3;q^4)$ and from \eqref{jIdentityA} we further get $j(q^3;q^4) = j(q;q^4)$. So we have
\begin{equation*}
-qj(q^{-1};q^4) = j(q;q^4)= J_{1,4}.
\end{equation*}
So we get the following expression
\begin{equation}
\begin{split}
\label{interm1}
f_{1,3,1}(q^2,q^2;q^2) &+ q^3f_{1,3,1}(q^6,q^6;q^2)\\
&=\frac{J_{4,8}J_{16,32}j(q^2;q^{16})j(q^{10};q^{16})J_{32}J_{1,4}}{j(-q^6;q^{16})j(-q^{14};q^{16})J_{16}^2}.
\end{split}
\end{equation}
By \eqref{jIdentityA} we first note that in the denominator 
$$j(-q^{14};q^{16})j(-q^6;q^{16}) = j(-q^2;q^{16})j(-q^{10};q^{16}).$$
Next we use \eqref{jxToProduct} where $x=q^2$ and $q=q^8$ to obtain
$$j(q^2;q^{16})j(q^{10};q^{16}) = j(q^2;q^8)\frac{J_{16}^2}{J_8},$$
and similarly we get
$$j(-q^2;q^{16})j(-q^{10};q^{16}) = j(-q^2;q^8)\frac{J_{16}^2}{J_8}.$$
After substituting into \eqref{interm1} and cancelling we obtain
\begin{equation}
\begin{split}
\label{interm2}
f_{1,3,1}(q^2,q^2;q^2) &+ q^3f_{1,3,1}(q^6,q^6;q^2)
=\frac{J_{4,8}J_{16,32}J_{2,8}J_{32}J_{1,4}}{\overline{J}_{2,8}J_{16}^2}.
\end{split}
\end{equation}

Now we use Proposition \ref{bigJidentities} to write
\begin{align*}
&J_{4,8} =\frac{J_4^2}{J_8}, &J_{16,32} = \frac{J_{16}^2}{J_{32}},\\
&J_{2,8} = \frac{J_2J_{8}}{J_{4}}, &\overline{J}_{2,8} = \frac{J_4^2}{J_2}.
\end{align*}
Substituting these into \eqref{interm2} yields
\begin{equation*}
\begin{split}
f_{1,3,1}(q^2,q^2;q^2) &+ q^3f_{1,3,1}(q^6,q^6;q^2)
=\frac{J_4^2J_{16}^2J_{2}J_{8}J_{32}J_2J_{1}J_4}{J_8J_{32}J_4J_4^2J_{16}^2J_2}.
\end{split}
\end{equation*}
Additional cancellation results in

\begin{theorem}We have
\begin{equation}
f_{1,3,1}(q^2,q^2;q^2) + q^3f_{1,3,1}(q^6,q^6;q^2) = J_1J_2.
\end{equation}
\end{theorem}

\section{The double-sums $f_{4,4,3}(\pm q^3,\pm q^2;q)$ }\label{section:Andrews-mock}
We have that
\begin{theorem} \label{theorem:f443} We have that
\begin{align}
    f_{4,4,3}(q^3,q^2;q)
    &=J_{1}\psi(q)\label{equation:f443-A}\\
    f_{4,4,3}(-q^3,-q^2;q)
    &=\frac{1}{4}\overline{J}_{0,3}\mu(q^3)
    -\frac{1}{2}\overline{J}_{1,3}\phi(q)+\Theta(q), \label{equation:f443-B}
\end{align}
where $\Theta(q)$ is a sum of quotients of theta functions, $\psi(q)$ is a third-order mock theta function, $\mu(q)$ is a second-order mock theta function, and $\phi(q)$ is a sixth-order mock theta function. 
\end{theorem}

For the definitions of the mock theta functions, we refer the reader to \cite[Section $5$]{HM} and references therein.

The proof of (\ref{equation:f443-A}) is trivial, but the proof of (\ref{equation:f443-B}) is very involved. We set up the pieces for the proof.  We begin with the following expression involving Appell--Lerch functions:
\begin{definition} Let $a,b,$ and $c$ be positive integers with  $D:=b^2-ac>0$.  Then
\begin{align*} 
G_{a,b,c}(x,y,z_1,z_0;q)
&:=\sum_{t=0}^{a-1}(-y)^tq^{c\binom{t}{2}}j(q^{bt}x;q^a)
m\Big (-q^{a\binom{b+1}{2}-c\binom{a+1}{2}-tD}\frac{(-y)^a}{(-x)^b},z_0;q^{aD}\Big ) \\
&\ \ \ \ \ +\sum_{t=0}^{c-1}(-x)^tq^{a\binom{t}{2}}j(q^{bt}y;q^c)
m\Big (-q^{c\binom{b+1}{2}-a\binom{c+1}{2}-tD}\frac{(-x)^c}{(-y)^b},z_1;q^{cD}\Big ).
\end{align*}
\end{definition}

The following theorem is crucial in the proof of (\ref{equation:f443-B}).
\begin{theorem}\cite[Theorem 4.2]{MZ22}\label{theo:general-fabc} Let $a,b,$ and $c$ be positive integers with  $D:=b^2-ac>0$. For generic $x$ and $y$, we have
\begin{align*}
& f_{a,b,c}(x,y;q)=G_{a,b,c}(x,y,-1,-1;q)+\frac{1}{j(-1;q^{aD})j(-1;q^{cD})}\cdot \theta_{a,b,c}(x,y;q),
\end{align*}
where
\begin{align*}
&\theta_{a,b,c}(x,y;q):=
\sum_{d^*=0}^{b-1}\sum_{e^*=0}^{b-1}q^{a\binom{d-c/2}{2}+b( d-c/2 ) (e+a/2  )+c\binom{e+a/2}{2}}(-x)^{d-c/2}(-y)^{e+a/2}\\
&\cdot\sum_{f=0}^{b-1}q^{ab^2\binom{f}{2}+\big (a(bd+b^2+ce)-ac(b+1)/2 \big )f} (-y)^{af}
\cdot j(-q^{c\big ( ad+be+a(b-1)/2+abf \big )}(-x)^{c};q^{cb^2})\\
&\cdot j(-q^{a\big ( (d+b(b+1)/2+bf)D +c(a-b)/2\big )}(-x)^{-ac}(-y)^{ab};q^{ab^2D})\\
&\cdot \frac{(q^{bD};q^{bD})_{\infty}^3j(q^{ D(d+e)+ac-b(a+c)/2}(-x)^{b-c}(-y)^{b-a};q^{bD})}
{j(q^{De+a(c-b)/2}(-x)^b(-y)^{-a};q^{bD})j(q^{Dd+c(a-b)/2}(-y)^b(-x)^{-c};q^{bD})}.
\end{align*}
Here $d:=d^*+\{c/2 \}$ and $e:=e^*+\{ a/2\}$, with  $0\le \{\alpha \}<1$ denoting fractional part of $\alpha$.
\end{theorem}

We have the following intermediate proposition, whose proof is immediate.
\begin{proposition} \label{prop:f443A} We have
\begin{align*}
    G_{4,4,3}(x,y,-1,-1;q)
    &=\sum_{t=0}^{3}(-y)^tq^{3\binom{t}{2}}j(q^{4t}x;q^4)
    m\Big (-q^{10-4t}\frac{y^4}{x^4},-1;q^{16}\Big ) \\
&\qquad +\sum_{t=0}^{2}(-x)^tq^{4\binom{t}{2}}j(q^{4t}y;q^3)
m\Big (q^{6-4t}\frac{x^3}{y^4},-1;q^{12}\Big ),
\end{align*}
and
\begin{align*}
\theta_{4,4,3}(x,y;q)
&=
\sum_{d=0}^{3}\sum_{e=0}^{3}q^{4\binom{d-1}{2}+4( d-1 ) (e+2) + 3\binom{e+2}{2}}(-x)^{d-1}(-y)^{e+2}\\
&\cdot\sum_{f=0}^{4}q^{64\binom{f}{2}+\big (16d+42+12e \big )f} y^{4f}
\cdot j(q^{3\big ( 4d+4e+8+16f \big )}x^{3};q^{48})\\
&\cdot j(-q^{4\big ( 4d+42+16f) \big)} x^{-12}y^{16};q^{256})
\cdot \frac{(q^{16};q^{16})_{\infty}^3j(-q^{ 4d+4e}x;q^{16})}
{j(q^{4e-2}x^4y^{-4};q^{16})j(-q^{4d+2}y^4x^{-3};q^{16})}.
\end{align*}
\end{proposition}
We then take the appropriate specializations to obtain
\begin{proposition}\label{prop:f443B} We have
\begin{align}
    G_{4,4,3}(-q^3,-q^2,-1,-1;q)
    &=\frac{1}{4}\overline{J}_{0,3}\mu(q^3)
    -\frac{1}{2}\overline{J}_{1,3}\phi(q)\label{equation:g443}\\
    &\qquad +\overline{J}_{1,3}+\overline{J}_{1,4}\Theta_{1}(q)
    +\overline{J}_{1,3}\Theta_{2}(q)
    +\frac{1}{4}\overline{J}_{0,3}\frac{J_{6,12}^2}{J_{3}^3},\notag
\end{align}
\end{proposition}

\begin{proposition}\label{prop:msplit} We have
\begin{align*}
    \sum_{t=0}^{3}q^{-\binom{t+1}{2}}m(-q^{6-4t},-1;q^{16})=0
\end{align*}
and
\begin{equation*}
    m(-q^{7},-1;q^{12})-q^{-1}m(-q,-1;q^{12})=m(q^2,-1;q^3)+\Theta_2(q),
\end{equation*}
where
\begin{equation*}
    \Theta_2(q):=\frac{J_{6}^3}{j(-q^2;q^3)j(-1;q^{12})}
    \sum_{r=0}^{1}\frac{q^{2r}j(-q^{7+3r};q^6)j(-q^{6r};q^{12})}
    {j(q^{7};q^{6})j(-q^{3r};q^{6})}.
\end{equation*}
We point out that there are no singularities in the denominators of the quotients of the theta functions.
\end{proposition}

\begin{proof}[Proof of Proposition \ref{prop:msplit}]
For the first identity we use (\ref{mrelation1}).  For the second identity we use \cite[Theorem $3.5$]{HM} with $n=2$, $q\to q^3$ $x=q^2$, $z=z^{\prime}=-1$.
\end{proof}

\begin{proof}[Proof of Proposition \ref{prop:f443B}]

Using the elliptic transformation property (\ref{jIdentityB}), we have
\begin{align*}
    G_{4,4,3}(-q^3,-q^2,-1,-1;q)
    &=\sum_{t=0}^{3}q^{3\binom{t}{2}+2t}j(-q^{4t+3};q^4)
    m\Big (-q^{6-4t},-1;q^{16}\Big ) \\
&\qquad +\sum_{t=0}^{2}q^{4\binom{t}{2}+3t}j(-q^{4t+2};q^3)
m\Big (-q^{7-4t},-1;q^{12}\Big )\\
&=\sum_{t=0}^{3}q^{-\binom{t+1}{2}}j(-q^{3};q^4)
    m\Big (-q^{6-4t},-1;q^{16}\Big ) \\
&\qquad +\sum_{t=0}^{2}q^{4\binom{t}{2}+3t}j(-q^{4t+2};q^3)
m\Big (-q^{7-4t},-1;q^{12}\Big ),
\end{align*}
and see that the first summand evaluates to zero by Proposition \ref{prop:msplit}.  

We consider the second summand.  Using identity (\ref{jIdentityB}), we have
\begin{align*}
    \sum_{t=0}^{2}&q^{4\binom{t}{2}+3t}j(-q^{4t+2};q^3)
m\Big (-q^{7-4t},-1;q^{12}\Big )\\
&=\overline{J}_{1,3}m(-q^7,-1;q^{12})
+\overline{J}_{0,3}m(-q^{3},-1;q^{12})
+q^{-2}\overline{J}_{1,3}m(-q^{-1},-1;q^{12})\\
&=\overline{J}_{1,3}m(-q^7,-1;q^{12})
+\overline{J}_{0,3}m(-q^{3},-1;q^{12})
-q^{-1}\overline{J}_{1,3}m(-q,-1;q^{12})
\end{align*}
Using Proposition \ref{prop:msplit}, we can rewrite the above as
\begin{align*}
    \sum_{t=0}^{2}&q^{4\binom{t}{2}+3t}j(-q^{4t+2};q^3)
m\Big (-q^{7-4t},-1;q^{12}\Big )\\
&=\overline{J}_{1,3}\Big ( m(q^2,-1;q^3)+\Theta_2(q)\Big ) 
+\overline{J}_{0,3}m(-q^{3},-1;q^{12}).
\end{align*}

From the compiled list of identities in \cite[Section $5$, (5.3)]{HM}, we have that
\begin{equation*}
    \mu(q^3)=4m(-q^3,-1;q^{12})-\frac{J_{6,12}^2}{J_{3}^3}.
\end{equation*}
Hence we can write
\begin{align*}
    \sum_{t=0}^{2}&q^{4\binom{t}{2}+3t}j(-q^{4t+2};q^3)
m\Big (-q^{7-4t},-1;q^{12}\Big )\\
&=\overline{J}_{1,3}\Big ( m(q^2,-1;q^3)+\Theta_2(q)\Big ) 
+\frac{1}{4}\overline{J}_{0,3}\Big (\mu(q^3)+\frac{J_{6,12}^2}{J_{3}^3} \Big). 
\end{align*}
We also have that \cite[Section $5$, (5.23)]{HM}
\begin{equation*}
    \phi(q)=2m(q,-1;q^3),
\end{equation*}
where $\phi(q)$ is sixth-order mock theta function.  From (\ref{mrelation2}) and (\ref{mrelation1}), we have
\begin{equation*}
    m(q^2,-1;q^3)= 1- q^{-1}m(q^{-1},-1;q^3)
    =1-m(q,-1;q^3)=1-\frac{1}{2}\phi(q).
\end{equation*}
As a result, we have that the second summand can be written
\begin{align*}
    \sum_{t=0}^{2}&q^{4\binom{t}{2}+3t}j(-q^{4t+2};q^3)
m\Big (-q^{7-4t},-1;q^{12}\Big )\\
&=\overline{J}_{1,3}\Big ( 1-\frac{1}{2}\phi(q) +\Theta_2(q)\Big ) 
+\frac{1}{4}\overline{J}_{0,3}\Big (\mu(q^3)+\frac{J_{6,12}^2}{J_{3}^3} \Big). \qedhere
\end{align*}
\end{proof}

\begin{proof}[Proof of Theorem \ref{theorem:f443}]  The proof of (\ref{equation:f443-A}) is trivial.  One simply consults  the compiled list of identities in \cite[Section $5$, (5.6)]{HM} and compares the $q$-hypergeometric forms.
For the proof of (\ref{equation:f443-B}), we use Proposition \ref{prop:f443B} to obtain
\begin{align*}
    f_{4,4,3}&(-q^3,-q^2;q)\\
    &=\frac{1}{4}\overline{J}_{0,3}\mu(q^3)
    -\frac{1}{2}\overline{J}_{1,3}\phi(q)\\
    &\qquad +\overline{J}_{1,3}
    +\overline{J}_{1,3}\Theta_{2}(q)
    +\frac{1}{4}\overline{J}_{0,3}\frac{J_{6,12}^2}{J_{3}^3}
    +\frac{1}{\overline{J}_{0,16}\overline{J}_{0,12}}\theta_{4,4,3}(-q^3,-q^2;q).
\end{align*}
Instead of trying to simplify $\theta_{4,4,3}(-q^3,-q^2;q)$, we just show that there are no singularities in the denominator.  Examining the two theta functions in the denominator leads to
\begin{align*}
    &j(q^{4e-2}(-q^3)^4(-q^2)^{-4};q^{16})j(-q^{4d+2}(-q^2)^4(-q^3)^{-3};q^{16}) = j(q^{4e+2};q^{16})j(q^{4d+1};q^{16})
\end{align*}
which never vanishes for $0<|q|<1$.
\end{proof}

\section*{acknowledgements}
We would like to thank Ae Ja Yee and Jonathan Bradley-Thrush for helpful comments and suggestions which improved the manuscript.   This research was supported by the Theoretical Physics and Mathematics Advancement Foundation BASIS, agreement No. 20-7-1-25-1.


\begin{thebibliography}{99}

\bibitem{A1} G. E. Andrews, {\em $q$-orthogonal polynomials, Rogers--Ramanujan identities, and mock theta functions}, Proc. Stek. Inst. of Math. {\bf 276} (2012), no. 1, 21-32.

\bibitem{ADH} G. E. Andrews, F. J. Dyson, D. R. Hickerson, {\em Partitions and indefinite quadratic forms}, Invent. Math. {\bf 91} (1988), no. 3, 391--407.

\bibitem{AW} G. E. Andrews, S. O. Warnaar {\em The Bailey transform and false theta functions}, Ramanujan J. {\bf 14} (2007), 173--188.

\bibitem{CK} S. H. Chan, B. Kim {\em On some double-sum false theta series}, J. Number Theory {\bf 190} (2018), 40--55.

\bibitem{He} E. Hecke, {\em \"Uber einen Zusammenhang zwischen elliptischen Modulfunktionen und indefiniten quadratischen Formen}, Mathematische Werke, Vandenhoeck and Ruprecht, G\"ottingen, (1959), pp. 418-427.

\bibitem{H1} D. R. Hickerson, {\em A proof of the mock theta conjectures}, Invent. Math {\bf 94}, (1988), 639--660.

\bibitem{H2} D. R. Hickerson, {\em On the seventh order mock theta functions}, Invent. Math {\bf 94}, (1988), 661-677.

\bibitem{HM} D. R. Hickerson, E. T. Mortenson, {\em Hecke-type double sums, Appell--Lerch sums, and mock theta functions, I}, Proc. London Math. Soc., (3) {\bf 109} (2014), no. 2, 382--422. 

\bibitem{KP} V. Kac, D. Peterson, {\em Infinite-Dimensional Lie Algebras, Theta Functions and Modular Forms}, Adv. Math. {\bf 53} (1984), 125--264.

\bibitem{L} Z. G. Liu, {\em On the $q$-derivative and $q$-series expansions}, Int. J. Number Theory Vol. 9 (2013), 2069-2089.

\bibitem{Kr1} L. Kronecker, {\em Zur Theorie der elliptischen Functionen}, Monastber. K. Adad. Wiss. Zu Berlin (1881) 1165-1172.

\bibitem{Kr2} L. Kronecker, {\em Leopold Kronecker's Werke}, Bd. IV (B.G. Teubner, Leipzig, 1929; reprinted by Chelsea, New York, 1968).

\bibitem{Mo21} E. T. Mortenson, {\em Hecke--Rogers double-sums and false theta functions}, Res. Number Theory {\bf 7} (2021), no. 2, Paper No. 28, 20 pp.

\bibitem{MZ22} E. T. Mortenson, S. Zwegers {\em The mixed mock modularity of certain duals of generalized quantum modular forms of Hikami and Lovejoy}, arXiv:2207.02591.

\bibitem{RLN} S. Ramanujan, {\em The Lost Notebook and Other Unpublished Papers}, Narosa Publishing House, New Delhi, 1988.

\bibitem{R} L. J. Rogers, {\em Second memoir on the expansion of certain infinite products}, Proc. London Math. Soc., {\bf 25} (1894), pp. 318-343.

\bibitem{WY} L. Wang, A. J. Yee, {\em Some Hecke--Rogers type identities}, Adv. Math. {\bf 349} (2019), 733--748.

\bibitem{Za1} D. B. Zagier, {\em Quantum modular forms}, in Quanta of Maths: Conference in honor of Alain Connes, Clay Mathematics Proceedings {\bf 11}, AMS and Clay Mathematics Institute $2010$, 658--675.

\bibitem{Zw} S. Zwegers, {\em Mock theta functions}, Ph.D. Thesis, Universiteit Utrecht, 2002.

\end{thebibliography}
\end{document}